%% file: main.tex
\documentclass[a4paper,12pt]{article}
\usepackage[english]{babel}
\usepackage{amsmath, amsfonts, amssymb, amsthm}
\usepackage{multirow}
\usepackage{graphicx}
\usepackage{fancyhdr}
\usepackage[a-1b]{pdfx}
\usepackage{hyperref}
\usepackage{float}
\usepackage{url}
\usepackage{pdfpages}
\usepackage[utf8]{inputenc}
\usepackage[top=1in, bottom=1.25in, left=1in, right=1 in]{geometry}
\usepackage{amssymb, amsmath}
\usepackage{bbm} 
\usepackage{xcolor}
\usepackage{graphicx}
\usepackage{adjustbox}
\usepackage{color}
\usepackage{mathtools}
\usepackage{comment}
\usepackage{enumitem} 
\usepackage{etoolbox} 
\usepackage{pgfplots} 
\usepackage{xspace} 

\input{commands.tex}

\providecommand{\keywords}[1]
{
  \small	
  \textbf{\textit{Keywords---}} #1
}

\title{An Anticipative Markov Modulated Market}
\author{Bernardo D'Auria
\thanks{D'Annunzio University of Chieti-Pescara, Economy Department, 65122 Pescara PE, Italy.\break\email{Email: bernardo.dauria@unich.it}.}
\and Jos\'{e} A. Salmer\'{o}n
\thanks{Carlos III University of Madrid, Statistics Department, 28911, Legan\'{e}s, Spain.\break\email{Email: joseantonio.salmeron@uc3m.es}.}
}

\numberwithin{equation}{section}
\newtheorem{Theorem}{Theorem}[section]
\newtheorem{Assumption}[Theorem]{Assumption}
\newtheorem{Example}[Theorem]{Example}
\newtheorem{Proposition}[Theorem]{Proposition}
\newtheorem{Definition}[Theorem]{Definition}
\newtheorem{Corollary}[Theorem]{Corollary}
\newtheorem{Lemma}[Theorem]{Lemma}
\newtheorem{Remark}[Theorem]{Remark}
\numberwithin{Theorem}{section}

\providecommand{\citep}[1]{(\cite{#1})}
\newcommand{\email}[1]{\href{mailto:#1}{#1}}

\pgfplotsset{compat=1.17}

\begin{document}
\maketitle
\begin{abstract}
A Markovian modulation captures the trend in the market and influences the market coefficients accordingly.
The different scenarios presented by the market are modeled as the distinct states of a discrete-time Markov chain.
In our paper, we assume the existence of such modulation in a market 
and,
as a novelty, we assume that it can be anticipative 
with respect to the future of the Brownian motion that drives the dynamics of the risky asset. 
By employing these own techniques of enlargement of filtrations,
we solve an optimal portfolio utility problem in both a complete 
and an incomplete market. 
Many examples of anticipative Markov chains are presented for which we compute the 
additional gain of the investor who has a more accurate information.
\end{abstract}


\input{content.tex}

\section*{Acknowledgments}
This research was partially supported by the Spanish \emph{Ministerio de Economía y Competitividad} grants MTM2017-85618-P (via FEDER funds) and 
PID2020-116694GB-I00.
The first author acknowledges financial support by the Community of Madrid  within the framework of the multi-year agreement with the Carlos III Madrid University in its line of action ``\emph{Excelencia para el Profesorado Universitario}'' (V Plan Regional de Investigación Científica e Innovación Tecnológica 2016-2020).
The second author acknowledges financial support by an FPU Grant (FPU18/01101) of Spanish \emph{Ministerio de Ciencia, Innovaci\'{o}n y Universidades}.

\bibliography{main.bib}

\end{document}

%% file: commands.tex
\newcommand{\cA}{\mathcal{A}}

\newcommand{\cE}{\mathcal{E}}
\newcommand{\cF}{\mathcal{F}}
\newcommand{\cG}{\mathcal{G}}
\newcommand{\cH}{\mathcal{H}}

\newcommand{\cY}{\mathcal{Y}}

\newcommand{\bE}{\mathbb{E}}
\newcommand{\bF}{\mathbb{F}}
\newcommand{\bG}{\mathbb{G}}

\newcommand{\bH}{\mathbb{H}}
\newcommand{\bOne}{\mathbbm{1}}

\newcommand{\bR}{\mathbb{R}}

\newcommand{\bV}{\mathbb{V}}

\newcommand{\PP}{\boldsymbol{P}}

\newcommand{\EE}{\boldsymbol{E}}
\newcommand{\VV}{\boldsymbol{V}}

\numberwithin{equation}{section}

\DeclareMathOperator*{\esssup}{ess\,sup}

\providecommand{\citep}[1]{(\cite{#1})}

\newcommand{\prT}[1]{(#1_t, \, 0 \leq{t} \leq{T})}
\newcommand{\prTmenos}[1]{(#1_t, \, 0 \leq{t} <{T})}

\newcommand{\Ito}{It\^{o}}

%% file: content.tex
\section{Introduction}\label{sec:intro}
We work within the framework of the optimal portfolio problem in a market in which a discrete-time Markov chain $\varepsilon=\prT{\varepsilon}$ is modulating the state of the \emph{market coefficients}.
The market is composed of two assets and by market coefficients we refer to the drift and the volatility of the risky asset and the short interest rate of the riskless one, see \eqref{def.SDE.DS} for a detailed definition.
The source of randomness driving the dynamics of the risky asset is the Brownian motion $W=\prT{W}$.
We consider, as one of the main novelties of the paper, 
that~$\varepsilon$ is anticipative with respect to the natural filtration of~$W$, denoted by~$\bF$. 
This is motivated by the fact that, in a financial market, the change of regime may be noticed by some investors before the rest, 
due to the presence of anticipative information that creates asymmetry, and, maybe, incompleteness.
In this context, we assume that there exists 
a trader that relies on the information provided by $\bG:=\bF\vee\bF^\varepsilon$, 
being $\bF^\varepsilon$ the natural filtration generated by $\varepsilon$.

The optimal portfolio problem consists of maximizing the expected utility of an investor's gains by looking at the optimal strategy within an admissible set.
A seminal paper about this problem is \cite{merton1969}, where a solution for the logarithmic and CRRA utilities is given for a complete market within the natural information flow $\bF$. 
In \cite{karatzas1991}, the utility function acquires a more general form and the problem is solved both for complete and incomplete market.

The existence of a Markov chain that modulates the market has been deeply studied.
It is often mentioned as a regime-switching model, which leads to a new scenarios where the market coefficients modify their trend for many reasons. 
A seminal paper is~\cite{Bauerle}, where the same market as in~\cite{merton1969} but with the modulation is considered.
Explicit solutions for the HARA utilities in the investment/consumption problem are obtained in~\cite{SotomayorCadenillas2009}.
In addition,
the problem can be stated with a semi-Markov modulation, for example in~\cite{GhoshGoswamiKumar2009}.
In~\cite{AltayColaneri2019}, they faced with a non-continuous risky asset, whose dynamics are driven by a Poisson process.
In~\cite{Azevedo2014}, a dynamic programming principle is derived for a 
Markov-switching jump–diffusion process
while in~\cite{Pamen2017}, 
the general modulated stochastic control problem is solved via maximum principle.

Market modulation is also considered in option pricing, see for example~\cite{Guo2001}. 
A special case is the American options, where an optimal stopping problem arise, see~\cite{Elliott2002} or~\cite{Perkowski2020}, 
the latter is motivated by a subsidy support problem.
Nowadays, the modulation of the market is still an interesting research topic, as in~\cite{Jin2020EJOR}, where a more complex and realistic optimal consumption/investment problem is formulated.

The goal of our paper is to consider an anticipative Markov chain,
to the best of our knowledge, it is never been attempted before.
The existing literature about the asymmetric information in the optimal portfolio framework is divided between two approaches.
The first one is the theory of enlargement of filtrations, both for the initial and the progressive types,
where the so-called \emph{Jacod's hypothesis} is assumed in order to guarantee the semimartingale decomposition of $W$ in the bigger filtration, see~\cite{karatzas1996} for a seminal reference. 
The second one is the anticipative calculus, 
where the stochastic integral dispenses the measurability assumptions,
we refer the monography~\cite{DiNunnoOksendalProske2009} where in the Chapter~16 they face with the optimal portfolio problem under multiple levels of information.
A paper linking both approaches is~\cite{IMKELLER2000}, 
where the semimartingale property is guaranteed by assuming a weaker absolutely continuous condition related to the Malliavin derivative of the conditional probability in an initial enlargement problem.

In this paper we aim to address the problem of asymmetric information in a modulated market. 
To the best of our knowledge, it was partially studied only by \cite{BaltasYannacopoulos2017}, 
leaving many questions still to be solved.
The authors assumed that the modulation is given by a diffusion process, whose driver is correlated to $W$, 
about which they have some future information.
In our work, we assume that the modulation is a discrete-time Markov chain and the information is about the immediate future of the risky asset, allowing a larger number of examples to be developed.

The paper is organized as follows. 
In Section~\ref{sec:anticip.modul.market} we start by defining the utility maximization problem and introduce the notation, 
we also state the Jacod's hypothesis. 
In Section~\ref{subsec:complete} we address the problem for a  complete market, i.e., where the Markov chain at any time is measurable in $\cF_T = \sigma(W_t,0\leq t \leq T)$.
We get the optimal solution in the enlarged filtration~$\bG$ 
by adapting the computations of~\cite{karatzas1991}.
In addition, we provide a number of novel examples for which we compute the $\bG$-semimartingale decomposition and the optimal portfolio under the logarithmic utility.
In Section~\ref{subsec:incomplete} we assume that the $\varepsilon$ also depends on another Brownian motion~$B$ that creates incompleteness in the market. 
We get the solution of the utility problem in the incomplete market in~$\bG$ and we show how to adapt the examples developed in Section~\ref{subsec:complete}.
\section{Anticipative Markov modulated market}\label{sec:anticip.modul.market}
As a general set-up we assume to work in a probability space
$(\Omega, \cF,\bF,\PP)$
where~$\cF$ is the event sigma-algebra, and $\bF=\{\cF_t, t\geq0\}$ is an augmented filtration that is generated by 
the natural filtration of the Brownian motion $W=\prT{W}$ and let $T>0$ be the horizon time.
Let~$\varepsilon$ be a discrete-time Markov chain whose state space is 
$E = \{0,1\}$ and denote by $\bF^\varepsilon$ its own filtration. 
The relation between $\bF$ and~$\bF^\varepsilon$ is unspecified a priori.
We assume that an agent can invest in a market composed just by two assets.
The first one is risk-less while the second one is a risky asset,
their dynamics are given by the following SDEs,
\begin{align}\label{def.SDE.DS}
    dD_t &= {D_t}\,r_{\varepsilon_t}\, dt\ , \quad D_0 = 1 \\
    dS_t &= {S_t}\left(\eta_{\varepsilon_t}\, dt + \xi_{\varepsilon_t}\, dW_t\right)\ ,\quad S_0 = s_0 >0\ ,
\end{align}
where we are considering the case that the drift, the volatility and the short interest rate can take only two different values, i.e.,
$\{\eta_0,\eta_1\}$, $\{\xi_0,\xi_1\}$ and $\{r_0,r_1\}$.
We assume that $\xi_0,\,\xi_1,\,r_0,\,r_1>0$ and $\xi_0\neq\xi_1$.
If we try to apply a version of the \Ito\, Lemma we get,
$$ \ln \frac{S_t}{s_0} = \int_0^t \left( \eta_{\varepsilon_s}-\frac{1}{2}\xi_{\varepsilon_s} \right) ds +\int_0^t \xi_{\varepsilon_s}dW_s \ ,\quad 0\leq t\leq T\ .$$
In this case, the stochastic integral is not well defined for any possible $\bF^\varepsilon$, so we need to give the correct interpretation of the stochastic integral
according to the measurability assumptions on $\varepsilon$.
We consider the following possibilities,
\begin{itemize}
    \item[a)] $\bF^\varepsilon\subset\bF$: In this case, the classical \Ito\, integral is well-defined and we recover a modulated version of the Merton problem, 
    solved in \cite{Bauerle}.
    \item[b)] $\varepsilon\perp W$: The independence of the processes guarantees that the classical \Ito\, integral is well-defined (see Section 3.3 of \cite{OksendalSDE}), 
    but the new source of randomness creates incompleteness.
    This problem can be considered as an special case of \cite{karatzas1991} and it has partially been solved in \cite{GhoshGoswamiKumar2009} or \cite{SotomayorCadenillas2009}.
    \item[c)] $\varepsilon_t\in\cF_T$ but $\varepsilon_t\not\in\cF_t$: In this case the process $\varepsilon$ is anticipative and the \Ito\, integral may be not well-defined.
    Usually in the literature the problem is solved by applying one of the following approaches.
    \begin{enumerate}
        \item \emph{Enlargement of filtrations:} 
        We need to derive the semimartingale decomposition of $W$ within the bigger filtration $\bG:=\bF\vee\bF^\varepsilon$ by imposing the so-called Jacod's hypothesis.
        \item \emph{Anticipative calculus:} We extend the framework of the integral in order to drop the need of the measurability assumptions.
        The new possibilities of integrals are the forward integral or the Skorohod integral.
    \end{enumerate}
    \item[d)] $\varepsilon_t\in\cG_T\supsetneq\cF_T$ but $\varepsilon_t\not\in\cF_t$: The anticipative issue still holds and in addition it is possible that market is incomplete.
\end{itemize}
In this paper we solve the problems formulated in $c)$ and $d)$ in the Sections~\ref{subsec:complete} and~\ref{subsec:incomplete}, respectively.
As we will argue the reason why, we exploit the properties of the binary state-space $E$ by using the enlargement of filtration techniques.

If we denote by \hbox{$X^\pi=\prT{X^\pi}$} the wealth of the portfolio of the investor under an admissible strategy~$\pi$,
we can write its dynamics as the solution of the following SDE, for $0 \leq t \leq T$,
\begin{align*}
\frac{dX_t^{\pi}}{X_t^{\pi}} &= (1-\pi_t)\frac{dD_t}{D_t} + \pi_t \frac{dS_t}{S_t}\ ,
\quad X_0=x_0 > 0 \ .
\end{align*}
Using~\eqref{def.SDE.DS}, it can be expressed in the following form
\begin{align}\label{X.SDE}
dX_t^{\pi} &= (1-\pi_t)X_t^{\pi}r_{\varepsilon_t} \, dt + \pi_t X_t^{\pi}\left(\eta_{\varepsilon_t} \, dt + \xi_{\varepsilon_t} \, dW_t\right) \ ,
\quad X_0 = x_0 > 0 \ .
\end{align}
Usually it is assumed that the strategy $\pi$ makes optimal use of all information at disposal of the agent at each instant,
and in general we are going to assume that the agent's flow of information, modeled by the filtration 
$\bH=\{\cH_t, 0\leq t \leq T\}$,
is possibly larger than filtration $\bF$, that is  $\bF \subset \bH$.
\begin{Definition}
We define the set of admissible strategies $\cA(\bH)$ which contains all the self-financing portfolios $\pi = \prT{\pi}$ such that,
\begin{itemize}\setlength\itemsep{.75em}
    \item $\pi$ is adapted with respect to the filtration $\bH$
    \item $ \int_0^T  \xi_{\varepsilon_t}^2\pi_t^2 dt < +\infty$, $\PP$-almost surely
\end{itemize}
where $\pi_t$ is the proportion of the capital invested in the risky asset at time t. 
\end{Definition}
We define the optimal portfolio~$\pi^\bH = \prT{\pi^\bH}$ 
as the solution of the following optimization problem,
\begin{align}\label{utility}
\bV_T^{\bH} := \sup_{\pi\in\cA({\bH})} \EE[U(X^{\pi}_{T})] = \EE\left[U\left(X^{\pi^\bH}_{T}\right)\right] \ .
\end{align}
The function $U:\bR^+\rightarrow \bR$ denotes the utility of the agent and it is assumed to be continuous, strictly increasing and strictly concave in its domain and satisfying~\hbox{$U(x)=-\infty$} for~$x<0$.
In addition we assume that it is continuously differentiable, being $U'$ also strictly increasing and strictly concave in its domain and satisfying the Inada conditions,
\begin{subequations}\label{inada.conditions}
\begin{align}
    U'(0)       &:= \lim_{x\to 0^+}U'(x) = + \infty \ ,\\
    U'(+\infty) &:= \lim_{x\to\infty}U'(x) = 0\ .
\end{align}
\end{subequations}
We will write $\bG$ as the minimum filtration which is right-continuous and contains  $\bF\vee\bF^\varepsilon$. 
As we consider a discrete-time Markov chain, we assume that it jumps at the set of times $\{t_0,t_1,\ldots,t_n\}$ with $t_n=T$. 
Then, the enlargement $\bF\subset\bG= \{\cG_t, 0\leq t \leq T\}$ is a sequence of initial enlargements 
given by
\begin{equation}\label{def.G}
    \cG_t = \bigcap_{s>t}\left( \cF_s \vee \sigma(\varepsilon_{t_0}, \varepsilon_{t_1}, \ldots, \varepsilon_{t_k}) \right)\ ,\quad t_k \leq t < t_{k+1}\ ,\quad k<n\ .
\end{equation}
So we can reduce it to the framework of initial enlargements in each interval $[t_k,t_{k+1})$.
Note that, as the quadratic variation of $S$,
$
\int_0^t S_s^2\xi^2_{\varepsilon_s} ds$, 
can be observed by the agent, the Markov chain $\varepsilon$ is also seen because $\xi_\cdot$ is bijective.
Then the natural information flow of an agent who knows about the modulation of $\varepsilon$ should be $\bG$ and the enlargement is justified according to the model given by \eqref{def.SDE.DS}.

In addition, \cite[Example 2.7]{AmendingerImkellerSchweizer1998} assures that the binary random variable always satisfy the following Jacod's hypothesis.
\begin{Assumption}\label{Assumpt.Jacod.abs}
The conditional distributions $\PP(\varepsilon_{t_k}\in\cdot\vert \cF_t)$ for $t\in[t_k,t_{k+1})$
almost surely verify the following absolutely continuity condition with respect to $\PP(\varepsilon_{t_k}\in\cdot)$,
$$\PP(\varepsilon_{t_k}\in\cdot\vert \cF_t) \ll \PP(\varepsilon_{t_k}\in\cdot)\ .$$
\end{Assumption}
This assumption assures the existence of a family of jointly measurable processes~$p^{e,k}$ defined as
\begin{equation*}
    p_t^{e,k} =  \frac{\PP( \varepsilon_{t_k} = e \vert \cF_t)}{\PP(\varepsilon_{t_k} = e)}\ ,
    \quad e\in\{0,1\}\ ,\quad k\in\{0,1,...,n-1\}\ ,\quad t_k\leq t< t_{k+1}\ . 
\end{equation*}
The crucial point of the filtration enlargement is to guarantee that the $\bF$-local martingales remain at least $\bG$-semimartingales, 
this is known in the literature as the \emph{(H') hypothesis}.
In \cite[Theorem~2.5]{Jacod1985}, it
is proven that the former Jacod's hypothesis implies the (H') hypothesis.
As the process~$p^{e,k}=(p^{e,k}_t,t_k\leq t < t_{k+1})$ is an~\hbox{$\bF$-martingale} for any $e\in \{0,1\}$, see~\cite[Lemma~2.1]{AmendingerImkellerSchweizer1998}, 
then 
\hbox{$\langle X,p^{e,k}\rangle^{\bF}$}
is well-defined for any $\bF$-local martingale $X$.
We introduce the notation,
\begin{equation}\label{G.bracket} 
\langle X,p^{\varepsilon_{t_k},k}\rangle_t^{\bF} 
:= \langle X,p^{\cdot,k}\rangle_t^{\bF}\circ \varepsilon_{t_k} \ ,\quad t_k\leq t < t_{k+1}\ .
\end{equation}
\begin{Proposition}\label{prop.jacod.semimartingale}
Let $X=\prT{X}$ be an $\bF$-local martingale and 
let~$\bG$ be as in \eqref{def.G}. 
Then,
\begin{equation}\label{Jacod.semimartingale}
    \widehat{X}_t = X_t - \int_{t_k}^t \frac{d\langle X,p^{\varepsilon_{t_k},k}\rangle_s^\bF}{p^{\varepsilon_{t_k},k}_{s-}}\ ,\quad t_k\leq t < t_{k+1}\ ,
\end{equation}
is a $\bG$-local martingale in the interval $[t_k,t_{k+1})$.
\end{Proposition}
\section{Complete market}\label{subsec:complete}
In order to get the completeness of the market, in this section we assume that
\hbox{$\varepsilon_{t_k}\in\cF_{t_{k+1}}$}, for any $k<n$, so there is no more source of randomness in the market than the one provided by the $\bF$-Brownian motion $W$.
By using the Jacod's hypothesis we guarantee the $\bG$-semimartingale decomposition of $W$.
A more explicit expression for the integral process appearing in \eqref{Jacod.semimartingale} can be obtained in the following result.
\begin{Theorem}\label{alpha.binary}
If we consider the enlargement defined in \eqref{def.G},
the $\bG$-adapted process
$$ \alpha_t^\varepsilon(k) = 
\EE[D_t \varepsilon_{t_k}\vert\cF_t]
\frac{\varepsilon_{t_k}-\EE[\varepsilon_{t_k}\vert \cF_t]}
{\VV[\varepsilon_{t_k}\vert \cF_t]} = \frac{\EE[D_t \varepsilon_{t_k}\vert\cF_t]}{\VV[\varepsilon_{t_k}\vert \cF_t]}
\int_{t}^{t_{k+1}} \EE[D_s \varepsilon_{t_k}\vert\cF_s]dW_s\ ,
$$
satisfies that $W_t - \int_{t_k}^t\alpha^\varepsilon_s(k) ds$ is a $\bG$-Brownian motion for $t_k\leq t < t_{k+1}$.
\end{Theorem}
\begin{proof}
It follows by applying 
\cite[Theorem 2]{dauriaSalmeron2021SAA} for binary random variables jointly with the Clark-Ocone formula
\cite[Theorem 6.35]{DiNunnoOksendalProske2009},
\begin{align}\label{eq.clark.ocone}
    \varepsilon_{t_k} = \EE[ \varepsilon_{t_k} ] + \int_0^{t_{k+1}} \EE[D_t \varepsilon_{t_k}\vert\cF_t] dt\ ,
\end{align}
where the Malliavin derivative is considered in the generalized white noise version for any $L^2(\PP)$-random variable $\cF_T$-measurable.
\end{proof}
The process $\alpha^\varepsilon(k)$ is strongly related with the $\bF$-local martingale $p^{e,k}$.
Indeed, in \cite[Remark 5]{baudoin2003} we can find the following relationship
\begin{equation}\label{eq.Baudoin.represent}
    dp_t^{e,k} = p_t^{e,k} \alpha_t^e(k) dW_t\ .
\end{equation}
By concatenating the previous theorem and using that
$\alpha^\varepsilon(k)\in L^2([t_k,t_{k+1}],dt\times\PP,\bG)$ for any $k<n$, 
we prove the following result in the whole interval time.
\begin{Corollary}\label{cor.BM.G}
In the context of Theorem \ref{alpha.binary},
$$ \widehat W_t := W_t - \int_0^t\alpha^\varepsilon_s ds \ ,\quad \alpha^\varepsilon_t :=  \sum_{k=0}^{n-1}\bOne_{\{t_k\leq t < t_{k+1}\}}\alpha^\varepsilon_t(k)\ , \quad 0\leq t \leq T\ .$$
is a well-defined $\bG$-Brownian motion in $[0,T]$. $\alpha^\varepsilon$ is named the \emph{information drift}.
\end{Corollary}
\begin{proof}
By applying the Jacod's hypothesis, 
see Assumption~\ref{Assumpt.Jacod.abs} and Theorem~\ref{alpha.binary},
we know that 
$\widehat W = \prTmenos{\widehat W}$ is a $\bG$-Brownian motion with $T$ excluded.
Moreover, as we are considering only enlargements with purely atomic random variables, 
by~\cite[Theorem 4.1]{AmendingerImkellerSchweizer1998}, we have that
$$ \frac{1}{2}\int_0^T\EE\left[\left(\alpha^\varepsilon_t\right)^2 \right]dt = -\sum_{i=1}^{n-1} H(\varepsilon_{t_k}) <+\infty\ , $$
being $H(\varepsilon_{t_k})$ the entropy operator of the random variable $\varepsilon_{t_k}$.
Then, we can appy \cite[Theorem 3.2]{Imkeller2003} and we get that $\widehat W$ is a $\bG$-Brownian motion with $T$ included.
\end{proof}
To make a consistent notation with the continuous-time processes, 
we sometimes consider 
$\varepsilon_t := \varepsilon_{t_k}$ for $t_k\leq t < t_{k+1}$, for example, in the SDE of the assets.
In the next lemma, following~\cite{karatzas1991}, we construct a $\bG$-exponential local martingale, usually called a \emph{deflator} in the literature, see~\cite[Section 1.4]{Aksamit2017} or~\cite[Definition 3.1]{Fontana2014}. 
We will require the definition of the following process
\begin{equation}\label{def.risk.process}
    \theta^{\varepsilon}_t := \frac{\eta_{\varepsilon_t} - r_{\varepsilon_t}}{\xi_{\varepsilon_t}}+\alpha_t^\varepsilon \ ,\quad 0 \leq t \leq  T\ .
\end{equation}
\begin{Lemma}
The process $D^{-1}ZS = (D_t^{-1}Z_tS_t,0 \leq t \leq T)$ is a non-negative $\bG$-local martingale,
where
\begin{equation}\label{def.emm}
    Z_t = \cE_t\left(-\int_0^\cdot \theta^{\varepsilon}_s d\widehat W_s\right) \ ,
\end{equation}
with $\mathcal{E}_t(\cdot)$ denoting the \emph{Dol\'eans-Dade exponential}.
Moreover, $D^{-1}Z X^{\pi}$ is also a non-negative $\bG$-local martingale for any $\pi\in\cA(\bG)$.
\end{Lemma}
\begin{proof}
The proof follows directly by applying the \Ito\, Lemma, a similar one can be consulted in~\cite[Proposition 3.4(2)]{AmendingerImkellerSchweizer1998}.
\end{proof}
In the following theorem we solve the utility maximization problem under the filtration $\bG$ in a complete market modulated by $\varepsilon$.
It is one of the main results of the paper, 
the proof employs techniques similar to those of \cite{Amendinger2003,karatzas1991, PhamQuenez2001} in the complete case.
\begin{Theorem}\label{theo.opt.complete}
The solution of the optimal portfolio problem is 
\begin{align}
  \sup_{\pi\in\cA(\bG)} 
  \EE[U( X^{\pi}_T)] =  \EE[(U\circ I)(\cY(x_0) D_T^{-1} Z_T)] \ ,
\end{align}
where $X^\pi$ satisfies~\eqref{X.SDE},
$I$ is the inverse of $U'$ and 
$\cY(x_0)$ is the measurable function satisfying
\begin{equation}
    \EE\left[ D_T^{-1} Z_T I\left( \cY(x_0) D_T^{-1} Z_T\right)\right] = x_0\ .
\end{equation}
The optimal portfolio is attained by a strategy {$\pi^{\bG}\in\cA(\bG)$} which hedges the random variable 
\hbox{$X_T^{\pi^{\bG}} = I(\cY(x_0) D_T^{-1} Z_T)$}.
\end{Theorem}
\begin{proof}
By the completeness of the market,
we know that there exists $\pi\in\cA(\bG)$ such that 
$X^{\pi}_T =  I(\cY(x_0) D_T^{-1} Z_T)$ 
$\PP$-almost surely.
By applying \cite[Lemma 6.2]{karatzas1991}, the process 
$D^{-1}ZX^{\pi}$ 
is a true \hbox{$\bG$-martingale}.
Therefore, denoting by
$X^{\widehat\pi}$ the wealth process for any strategy $\widehat\pi\in\cA\left(\bG\right)$ and using the properties of the concave functions, we get
$$ U\left(X_T^{\pi}\right) \geq U(X_T^{\widehat\pi}) + \cY(x_0)D_T^{-1}Z_T\left(X_t^{\pi}-X_T^{\widehat \pi}\right)\ .$$
Taking expectation in both sides we have,
\begin{align*}
    \EE\left[ U\left(X_T^{\pi}\right)\right] &\geq \EE[U(X_T^{\widehat\pi})] + \EE\left[\cY(x_0) D_T^{-1}Z_T \left(X_T^{\pi}-X_T^{\widehat \pi}\right)\right]\\ 
    &= \EE[U(X_T^{\widehat\pi})] + \EE\left[\cY(x_0) D_T^{-1}Z_T \left(X_T^{\pi}-X_T^{\widehat \pi}\right)\right] \\
    &\geq \EE[U(X_T^{\widehat\pi})] + \cY(x_0)\left(x_0-\EE[ D_T^{-1}Z_T X_T^{\widehat \pi}]\right) \geq \EE[U(X_T^{\widehat\pi})]\ ,
\end{align*}
where, in the last step, we used that $D^{-1} Z X^{\pi}$ and 
$D^{-1} ZX^{\widehat\pi}$ are respectively a \hbox{$\bG$-martingale} and a \hbox{$\bG$-super} martingale, 
we finally denoted by $\pi^\bG$ the optimal strategy.
\end{proof}
\begin{Remark}
If the process $\alpha^\varepsilon$ satisfies the \emph{Novikov's condition}, i.e.,
$$ \EE\left[\exp\left(\frac{1}{2} \int_0^T\left(\alpha^\varepsilon_t\right)^2dt\right)\right]<+\infty\ , $$
then \emph{(NFLVR)}$_{\bG}$ holds true and the deflator defines an \emph{Equivalent Local Martingale Measure (ELMM)}, 
see \cite[Proposition 3.2]{Fontana2014}.
An example of enlargement satisfying this condition can be found in~\cite[Theorem 4.16]{DauriaSalmeron2020}.
\end{Remark}
In the following result, 
we give the value of the additional information of an agent playing with $\bG$ versus another who is totally unaware of the existence of such modulation.
\begin{Proposition}\label{prop.price}
If we consider the enlargement $\bF\subset\bG$, then the price of the information given by $\varepsilon$ under logarithmic utility is
\begin{align}\label{eq.gains.B}
    \bV_T^\bG-\bV_T^\bF =& -\sum_{k=0}^{n-1}\left(\PP(\varepsilon_{t_k}=1)\ln \PP(\varepsilon_{t_k}=1) + \PP(\varepsilon_{t_k}=0)\ln \PP(\varepsilon_{t_k}=0)\right)\notag\\ 
&+ (M_1-M_0)\int_0^T\EE[D_t \varepsilon_{t}] dt\ ,
\end{align}
where we have defined 
$M_1:=\dfrac{\eta_1-r_1}{\xi_1}\ ,\ M_0:=\dfrac{\eta_0-r_0}{\xi_0}\ .$
\end{Proposition}
\begin{proof}
As we consider $U(x) = \ln x$, from Theorem \ref{theo.opt.complete} we conclude that 
$$ X_T^{\pi^{\bG}} = x_0D_TZ_T^{-1}\ ,\quad \pi^\bG_t = \frac{\theta_t^\varepsilon}{\xi_{\varepsilon_t}} \ ,\quad 0\leq t \leq T\ . $$
The gains under the filtration $\bG$ at $T$ are
\begin{align*}
    \bV_T^\bG =& \EE\left[\ln X_T^{\pi^{\bG}}\right] = 
    \ln x_0 + \int_0^T \EE[r_{\varepsilon_t}] dt + \EE\left[\int_0^T \theta_t^\varepsilon d\widehat W_t\right] + \frac{1}{2}\int_0^T \EE\left[(\theta_t^{\varepsilon})^2\right]dt\\
    =& \ln x_0 + \int_0^T \EE[r_{\varepsilon_t}] dt  + \frac{1}{2}\int_0^T \EE\left[\left(\frac{\eta_{\varepsilon_t} - r_{\varepsilon_t}}{\xi_{\varepsilon_t}}\right)^2+\left(\alpha_t^\varepsilon \right)^2 + 2\frac{\eta_{\varepsilon_t} - r_{\varepsilon_t}}{\xi_{\varepsilon_t}}\alpha_t^\varepsilon \right]dt\\
    =& \bV_T^\bF  + \frac{1}{2}\int_0^T \EE\left[\left(\alpha_t^\varepsilon \right)^2\right]dt  +\int_0^T \EE\left[\frac{\eta_{\varepsilon_t} - r_{\varepsilon_t}}{\xi_{\varepsilon_t}}\alpha_t^\varepsilon \right]dt\ ,
\end{align*}
where we have used that 
$\theta^\varepsilon\in L^2(dt\times\PP,\bG)$ in order to conclude the zero expectation of the \Ito\, integral.
In order to rewrite the last term, we consider 
$$\frac{\eta_{\varepsilon_t}-r_{\varepsilon_t}}{\xi_{\varepsilon_t}} = (M_1-M_0)\varepsilon_{t_k} + M_0\ ,\quad t_k\leq t < t_{k+1}\ ,$$
with the previous definition of $M_1,M_0$.
Then we get,
\begin{align*}
    \int_0^T&\EE\left[\frac{\eta_{\varepsilon_t} - r_{\varepsilon_t}}{\xi_{\varepsilon_t}}\alpha_t^\varepsilon \right] dt 
    = \sum_{k=0}^{n-1}\int_{t_k}^{t_{k+1}}\EE\left[\frac{\eta_{\varepsilon_t} - r_{\varepsilon_t}}{\xi_{\varepsilon_t}}\alpha_t^\varepsilon(k) \right] dt\\
    =& \sum_{k=0}^{n-1}\int_{t_k}^{t_{k+1}}\EE\left[\left((M_1-M_0)\varepsilon_{t_k} + M_0\right)\alpha_t^\varepsilon(k) \right] dt\\
    =& \sum_{k=0}^{n-1}\int_{t_k}^{t_{k+1}}\EE\left[\left((M_1-M_0)\varepsilon_{t_k} + M_0\right)\frac{\EE[D_t \varepsilon_{t_k}\vert\cF_t]}{\VV[\varepsilon_{t_k}\vert \cF_t]}
    \int_{t}^{t_{k+1}} \EE[D_s \varepsilon_{t_k}\vert\cF_s]dW_s \right] dt\\
    =& \sum_{k=0}^{n-1}\int_{t_k}^{t_{k+1}}\EE\left[(M_1-M_0)\varepsilon_{t_k}\frac{\EE[D_t \varepsilon_{t_k}\vert\cF_t]}{\VV[\varepsilon_{t_k}\vert \cF_t]}
    \int_{t}^{t_{k+1}} \EE[D_s \varepsilon_{t_k}\vert\cF_s]dW_s \right] dt\\
    =& \sum_{k=0}^{n-1}\int_{t_k}^{t_{k+1}}\EE\left[(M_1-M_0)\left(\EE[\varepsilon_{t_k}] 
    + \int_{0}^{t_{k+1}}\EE[D_s\varepsilon_{t_k}\vert\cF_s]dW_s\right)\times\right.\\
    &\left. \times \frac{\EE[D_t \varepsilon_{t_k}\vert\cF_t]}{\VV[\varepsilon_{t_k}\vert \cF_t]}
    \int_{t}^{t_{k+1}} \EE[D_s \varepsilon_{t_k}\vert\cF_s]dW_s \right] dt\\
    =& \sum_{k=0}^{n-1}\int_{t_k}^{t_{k+1}}\EE\left[(M_1-M_0) \frac{\EE[D_t \varepsilon_{t_k}\vert\cF_t]}{\VV[\varepsilon_{t_k}\vert \cF_t]}
    \left(\int_{t}^{t_{k+1}} \EE[D_s \varepsilon_{t_k}\vert\cF_s]dW_s\right)^2 \right] dt\\
    =& \sum_{k=0}^{n-1}\int_{t_k}^{t_{k+1}}\EE\left[(M_1-M_0) \frac{\EE[D_t \varepsilon_{t_k}\vert\cF_t]}{\VV[\varepsilon_{t_k}\vert \cF_t]}\EE\left[
    \int_{t}^{t_{k+1}}\left( \EE[D_s \varepsilon_{t_k}\vert\cF_s]\right)^2ds\vert\cF_t\right] \right] dt\\
    =& (M_1-M_0)\sum_{k=0}^{n-1}\int_{t_k}^{t_{k+1}}\EE\left[\EE[D_t \varepsilon_{t_k}|\cF_t]\right] dt  = (M_1-M_0)\int_0^T\EE[D_t \varepsilon_{t}] dt \ , 
\end{align*}
where in the third equality we have used Theorem~\ref{alpha.binary}, in the fifth one the equation~\eqref{eq.clark.ocone} and in the eighth one we use that
\begin{align*}
    \VV[\varepsilon_{t_k}\vert \cF_t] =& \EE[(\varepsilon_{t_k}-\EE[\varepsilon_{t_k}\vert \cF_t])^2\vert\cF_t]  =  
    \EE\left[\left( 
    \int_{t}^{t_{k+1}} \EE[D_s \varepsilon_{t_k}\vert\cF_s]dW_s
    \right)^2\vert\cF_t\right]\\
    =& \EE\left[
    \int_{t}^{t_{k+1}}\left( \EE[D_s \varepsilon_{t_k}\vert\cF_s]\right)^2ds\vert\cF_t\right]\ .
\end{align*}
The rest of the equalities uses basic \Ito\,integral properties.
Finally, we compute
\begin{align*}
    \EE\left[\left(\alpha_t^\varepsilon \right)^2\right] &=
    \EE\left[\left(\sum_{k=0}^{n-1}\bOne_{\{t_k\leq t < t_{k+1}\}}\alpha^\varepsilon_t(k) \right)^2\right] = 
    \sum_{k=0}^{n-1}\bOne_{\{t_k\leq t < t_{k+1}\}}\EE\left[\left(\alpha^\varepsilon_t(k) \right)^2\right]
\end{align*}
and by applying 
\cite[Theorem 4.1]{AmendingerImkellerSchweizer1998},
we get the result.
\end{proof}
\begin{Remark}\label{rk.gains}
Although we assume that $\varepsilon$ must be $\cF_T$-measurable in order to use the Clark-Ocone formula \eqref{eq.clark.ocone}, 
in \cite{Fontana18PRP} it was shown that the predictable representation holds in initially enlarged filtrations satisfying the Jacod's hypothesis, 
so it can be trivially extended to these filtrations by considering the non-anticipative derivative instead the process $\EE[D_t\varepsilon_t\vert\cF_t]$ and taking the $\cF_t$-conditional version of \cite[Theorem 4.1]{Fontana18PRP}.
By using \cite[Lemma 5]{dauriaSalmeron2021SAA}, 
the following representation
$$ \int_0^T\EE[D_t \varepsilon_{t}] dt 
    = \PP(\varepsilon_{t_k}=1)\sum_{k=0}^{n-1}\int_{t_k}^{t_{k+1}}
    \EE\left[\alpha^\varepsilon_t(k)p_t^{\varepsilon,k}\right] dt $$
also holds.
\end{Remark}
\begin{Example}\label{Example.BM.increment}
We define 
$$ \varepsilon_{t_k} = \bOne_{\{W_{t_{k+1}} > W_{t_k}\}}\ ,\quad k\in\{0,1,...,n-1\}  \ .$$
Using the independent increments of $W$, 
it can be verified that $\varepsilon$ is a Markov process 
and we provide an explicit expression for the information drift in the next lemma.
\begin{Lemma}
    $$\alpha^e_t(k) := \frac{(-1)^{1+e}}{\sqrt{t_{k+1}-t}} \frac{\Phi'\left(\frac{W_{t_k}-W_t}{\sqrt{t_{k+1}-t}}\right)}{\PP(\varepsilon_{t_k} = e|W_t)}\ ,\quad e\in\{0,1\}\ ,\quad t_k\leq t < t_{k+1}\ ,$$
    where $\Phi(x):=\int_{-\infty}^x\frac{\exp(-y^2/2)}{\sqrt{2\pi}}dy$.
\end{Lemma}
\begin{proof}
It follows directly by \cite[Remark 6]{dauriaSalmeron2021SAA}.
\end{proof}
By applying the \Ito\, Lemma to \eqref{X.SDE} we obtain
$$\ln X_T^\pi = \ln x_0 + \int_0^T r_{\varepsilon_s} + \pi_s\left(\eta_{\varepsilon_k} - r_{\varepsilon_s} + \xi_{\varepsilon_s}\alpha^\varepsilon_s - \frac{1}{2}\xi^2_{\varepsilon_s} \pi_s \right) ds + \int_0^T \xi_{\varepsilon_s}\pi_s d\widehat W_s\ . $$
As $\widehat W$ is a $\bG$-Brownian motion, the expectation of the utility is reduced to
$$\EE\left[\ln X_T^\pi\right] = \ln x_0 + \int_0^T \EE\left[r_{\varepsilon_s} + \pi_s\left(\eta_{\varepsilon_k} - r_{\varepsilon_s} + \xi_{\varepsilon_s}\alpha^\varepsilon_s - \frac{1}{2}\xi^2_{\varepsilon_s} \pi_s \right) \right] ds \ . $$
Following the lines of \cite[Theorem 16.54]{DiNunnoOksendalProske2009}, 
 we maximize for each \hbox{$(s,\omega)\in[0,T]\times\Omega$} the functional
    $$ J(\pi) := r_{\varepsilon_s} + \pi\left(\eta_{\varepsilon_s} - r_{\varepsilon_s} + \xi_{\varepsilon_s}\alpha^\varepsilon_s - \frac{1}{2}\xi^2_{\varepsilon_s} \pi \right) \ . $$
    Then, a stationary point of $J$ satisfies
    $0 = J'(\pi) = \eta_{\varepsilon_s} - r_{\varepsilon_s} + \xi_{\varepsilon_s}\alpha^\varepsilon_s - \xi^2_{\varepsilon_s} \pi ,$
   and we obtain the candidate
    \begin{equation}
        \pi^\bG_t =  \frac{\eta_{\varepsilon_t}}{\xi^2_{\varepsilon_t}} +\frac{\alpha^{\varepsilon}_t}{\xi_{\varepsilon_t}} =
        \sum_{k=0}^{n-1}\bOne_{\{k\leq t < k+1\}}\left( \frac{\eta_{\varepsilon_{t_k}}}{\xi^2_{\varepsilon_{t_k}}} +\frac{\alpha^{\varepsilon}_{t_k}(k)}{\xi_{\varepsilon_{t_k}}} \right)\ .
    \end{equation}
    Since $J$ is concave, the strategy $\pi^\bG$ is a maximum of the optimization problem.
\end{Example}
\begin{Example}\label{Ex.max}
Let $M=\prT{M}$ be the running maximum of $W$, i.e., 
$$M_t:=\sup_{0\leq s\leq t} W_s\ .$$
We use the Markov property of the process $M_t-W_t$ in order to define the following Markov chain $\varepsilon$,
\begin{equation}\label{ex.drawdown}
    \varepsilon_{t_k} = \bOne_{ \{M_{t_{k+1}}-W_{t_{k+1}}> c \}}  \ ,\quad c\in\bR^+\ ,\quad k\in\{0,1,...,n-1\}\ .
\end{equation}
As we have seen in the Example \ref{Example.BM.increment}, 
the optimal strategy under the logarithmic utility is fully determined by the information drift $\alpha^\varepsilon$, 
which is computed in the next lemma.
    \begin{Lemma}
       $$\alpha^e_t(k) := 
       \frac{(-1)^{1+e}}{\sqrt{t_{k+1}-t}}
       \frac{\Phi'\left(\frac{c+M_t-W_t}{\sqrt{t_{k+1}-t}}\right) - 
             \Phi'\left(\frac{c-M_t+W_t}{\sqrt{t_{k+1}-t}}\right)}{\PP(\varepsilon_{t_k} = e|\cF_t)}\ ,\quad e\in\{0,1\}\ ,\quad t_k\leq t < t_{k+1}\ . $$
    \end{Lemma}
    \begin{proof}
    We proceed as before, by computing 
        \begin{equation*}
            p_t^{1,k} = \frac{\PP(M_{t_{k+1}}-W_{t_{k+1}} > c |\cF_t)}{\PP(M_{t_{k+1}}- W_{t_{k+1}} > c)} \ ,
        \end{equation*}
        we operate separately with the numerator, looking at a more explicit expression. 
        We denote by  
        $M_{s,t}:=\sup_{u\in[s,t]} W_u$ and $a\vee b:=\max\{a,b\}$, then
        \begin{align*}
            \PP(M_{t_{k+1}}-& W_{t_{k+1}} > c |\cF_t) = \PP(M_t\vee M_{t,t_{k+1}}\, -W_{t_{k+1}} > c |\cF_t)\\ 
            =& \PP(M_t-W_{t_{k+1}} > c ,\, M_t>M_{t,t_{k+1}}|\cF_t)\\
            &+\PP(M_{t,t_{k+1}}-W_{t_{k+1}} > c ,\, M_t\leq M_{t,t_{k+1}}|\cF_t)\\
            =& \PP(M_t-W_t-(W_{t_{k+1}}-W_t) > c ,\, M_t-W_t>M_{t,t_{k+1}}-W_t|\cF_t)\\
            &+\PP(M_{t,t_{k+1}}-W_t-(W_{t_{k+1}}-W_t) > c ,\, M_t-W_t\leq M_{t,t_{k+1}}-W_t|\cF_t)\\
            =& \PP(M_t-W_t-(W_{t_{k+1}}-W_t) > c ,\, M_t-W_t>M_{t,t_{k+1}}-W_t)\\
            &+\PP(M_{t,t_{k+1}}-W_t-(W_{t_{k+1}}-W_t) > c ,\, M_t-W_t\leq M_{t,t_{k+1}}-W_t)\ .
        \end{align*}
        We define $K_t=M_t-W_t$, $X=M_{t,t_{k+1}}-W_t$ and $Y=W_{t_{k+1}}-W_t$.
        In~\cite[Proposition~8.1]{karatzas1991brownian} it is stated that the vector $(X,Y)$ has the following density function,
        $$p(x,y,t_{k+1}-t) =\frac{2(2x-y)}{\sqrt{2\pi(t_{k+1}-t)^3}} \exp\left(-\frac{(2x-y)^2}{2(t_{k+1}-t)}\right)\ ,\quad x> y\vee 0\ . $$
        Then,
        \begin{align*}
             \PP(M_{t_{k+1}}-W_{t_{k+1}} > c |\cF_t) =& \PP(Y < {K_t}-c ,\, X<{K_t}) + \PP(X-Y > c ,\, X\geq {K_t})\\
             =&\int_0^{K_t}\int_{-\infty}^{{K_t}-c} p(x,y,t_{k+1}-t) dydx\\ 
             &+ \int_{K_t}^{+\infty}\int_{-\infty}^{x-c} p(x,y,t_{k+1}-t) dydx\\
            =& 2\int_0^{K_t}  \frac{\Phi'\left( \frac{2x-{K_t}+c}{\sqrt{t_{k+1}-t}} \right)}{\sqrt{t_{k+1}-t}} dx+ 2\int_{K_t}^{+\infty}\frac{\Phi'\left( \frac{x+c}{\sqrt{t_{k+1}-t}} \right)}{\sqrt{t_{k+1}-t}} dx \\
            =& \int_{c-K_t}^{c+K_t}  \frac{\Phi'\left( \frac{z}{\sqrt{t_{k+1}-t}} \right)}{\sqrt{t_{k+1}-t}} dz+ 2\int_{c+K_t}^{+\infty}\frac{\Phi'\left( \frac{z}{\sqrt{t_{k+1}-t}} \right)}{\sqrt{t_{k+1}-t}} dz \\
            =& \overline{\Phi}\left(\frac{c-K_t}{\sqrt{t_{k+1}-t}}\right) + 
             \overline{\Phi}\left(\frac{c+K_t}{\sqrt{t_{k+1}-t}}\right) \ .
        \end{align*}
        By defining
        $f(t,K_t):= \PP(M_{t_{k+1}}-W_{t_{k+1}} > c |\cF_t)$
        we aim to compute $df(t,K_t)$ via \Ito\, Lemma.
        We note that the partial derivatives satisfy the following
        \begin{align*}
            \frac{\partial}{\partial K}f(t,K_t) &= \frac{1}{\sqrt{t_{k+1}-t}}\Phi'\left(\frac{c-K_t}{\sqrt{t_{k+1}-t}}\right)            -\frac{1}{\sqrt{t_{k+1}-t}}\Phi'\left(\frac{c+K_t}{\sqrt{t_{k+1}-t}}\right)\\
           0 &= \frac{\partial}{\partial t}f(t,K_t)+\frac{1}{2} \frac{\partial^2}{\partial K^2}f(t,K_t) \ ,
        \end{align*}
        and we conclude
        $$ df(t,K_t) =
        \frac{\partial}{\partial K}f(t,K_t) dK_t \ .$$
        Note that the process $M$ is increasing only in the set $\{t,\,K_t = 0\}$ and then we get 
        \hbox{$\frac{\partial}{\partial K}f(t,0)dM_t = 0$}. 
        We have proved the following relationship
         $$ df(t,K_t) 
         = - \frac{\partial}{\partial K}f(t,K_t) dW_t = \frac{\Phi'\left(\frac{c+K_t}{\sqrt{t_{k+1}-t}}\right) - \Phi'\left(\frac{c-K_t}{\sqrt{t_{k+1}-t}}\right)}{\sqrt{t_{k+1}-t}} dW_t\ . $$
         We conclude the following representation
         $$ dp_t^{1,k} = p_t^{1,k} \frac{\Phi'\left(\frac{c+K_t}{\sqrt{t_{k+1}-t}}\right) - \Phi'\left(\frac{c-K_t}{\sqrt{t_{k+1}-t}}\right)}{\sqrt{t_{k+1}-t}\,\PP(\varepsilon_{t_k} =  1 |\cF_t)} dW_t $$
        and by using \eqref{eq.Baudoin.represent} the result follows when $e=1$.
        We repeat the same reasoning when $e=0$.
    \end{proof}
\end{Example}
\begin{Example}\label{Ex.pathwise}
        We consider the enlargement given by the Markov chain
        $$ \varepsilon_{t_k} = \bOne_{\{W_s\leq B_k\, ,\, t_k\leq s < t_{k+1}\}}  \ ,\quad k\in\{0,1,...,n-1\}\ .$$
        We assume that $B_k$ is $\cF_{t_k}$-measurable and it can be different for any $k$. 
        Then, the random variables $\{\varepsilon_{t_k}\}_k$ are independent.
        This example has been introduced in \cite{baudoin2003}, in the Section 3.3 by the name of 
        \emph{Pathwise Conditioning}.
        \begin{Lemma}
        $$\alpha_t^\varepsilon(k) :=  \left \{ 
        \begin{matrix} 
        -f_{t,t_{k+1}}(B_k)
      \frac{\varepsilon_{t_k}-\PP(\varepsilon_{t_k}=1|\cF_t)} {\PP(\varepsilon_{t_k}=1|\cF_t)\PP(\varepsilon_{t_k}=0|\cF_t)} & \mbox{ if } M_{t_k,t}\leq B_k\\ 
      \\
        0 & \mbox{ if } M_{t_k,t} > B_k
        \end{matrix}
        \right. $$
      where we compute $\PP(\varepsilon_{t_k}=0|\cF_t) = 1 - \PP(\varepsilon_{t_k}=1|\cF_t)$ and 
      \begin{align*}
          f_{t,t_{k+1}}(m) &= \frac{2}{ \sqrt{t_{k+1}-t}}\Phi'\left(\frac{m-W_t}{\sqrt{t_{k+1}-t}}\right)\ ,\quad m\geq W_t\\
          \PP(\varepsilon_{t_k}=1|\cF_t) &=  \bOne_{\{M_{t_k,t}\leq B_k\}} \int_{W_t}^{B_k} f_{t,t_{k+1}}(m)dm \ .
      \end{align*}
        \end{Lemma}
        \begin{proof}
        We proceed by applying the Theorem \ref{alpha.binary}. 
        Note that $\varepsilon$ can be expressed in terms of the running maximum of $W$ as follows,
        \hbox{$\varepsilon_{t_k} = \bOne_{\{M_{t_k,t_{k+1}}\leq B_k\}}.$}
        Then, we compute the Malliavin derivative 
        by applying \cite[Corollary 5.3]{Bermin02}, 
        as it is argued in the Example 5.3 of the same reference, we have
        \begin{equation*}
            D_t \bOne_{\{M_{t_k,t_{k+1}}\leq B_k\}} = -\delta_{B_k}(M_{t_k,t_{k+1}}) \bOne_{\{ M_{t_k,t}\leq M_{t,t_{k+1}} \}}\ .
        \end{equation*}
        And the conditional expectation is computed as follows,
        \begin{align*}
            \EE[D_t \bOne_{\{M_{t_k,t_{k+1}}\leq B_k\}}|\cF_t] &= -\EE[\delta_{B_k}(M_{t_k,t_{k+1}})  \bOne_{\{ M_{t_k,t}\leq M_{t,t_{k+1}} \}}|\cF_t] \\
            &= -\EE[\delta_{B_k}(M_{t,t_{k+1}})  \bOne_{\{ M_{t_k,t}\leq M_{t,t_{k+1}} \}}|\cF_t]\\
            &= -\int_{M_{t_k,t}}^{+\infty} \delta_{B_k}(m) f_{t,t_{k+1}}(m)dm = -\bOne_{\{M_{t_k,t}\leq B_k\}}\, f_{t,t_{k+1}}(B_k)\ ,
        \end{align*}
        where $f_{t,t_{k+1}}$ is the $\cF_t$-conditional density function of $M_{t,t_{k+1}}$.
        To end the example, we compute the conditional probabilities as follows,
        \begin{align*}
            \PP(\varepsilon_{t_k}=1|\cF_t) =& \PP(M_{t_k,t_{k+1}}\leq B_k|\cF_s) = \PP(M_{{t_k},t}\leq B_k, M_{t,t_{k+1}}\leq B_k|\cF_t)\\
            =& \bOne_{\{M_{{t_k},t}\leq B_k\}}\PP( M_{t,t_{k+1}}\leq B_k|\cF_t)
            = \bOne_{\{M_{{t_k},t}\leq B_k\}} \int_{W_t}^{B_k} f_{t,t_{k+1}}(m) 
            dm\ .
        \end{align*}  
        \end{proof}
\end{Example}
\begin{Example}
        We consider a sequence  $L_0,L_1,...,L_n$ of $\cF_T$-independent
        binary random variables with parameters $p_0,p_1,...,p_n$,
        and we define the Markov chain
        $$ \widetilde\varepsilon_{t_k} = L_k\varepsilon_{t_k}\ , \quad k\in\{0,1,...,n\}\ ,$$
        being $\varepsilon_{t_k}$ any of the Examples 
        \ref{Example.BM.increment}, 
        \ref{Ex.max} and \ref{Ex.pathwise}.
        In this case the information of the agent is disturbed by some additional source of noises, provided by the binary random variables $\{L_j\}_{j=1}^n$.
        Despite the presence of such noise, not $\cF_T$-measurable,
        the market maintains its completeness, 
        as proved in \cite[Proposition 3.1]{Fontana18PRP} for initial enlargements satisfying the Jacod's hypothesis.
        Then, we compute
        \begin{align}\label{eq.p.disturb}
             \widetilde p_t^{1,k} &= \frac{\PP(\widetilde\varepsilon_{t_k} = 1|\cF_t)}{\PP(\widetilde\varepsilon_{t_k} = 1)} = \frac{\EE\left[\widetilde\varepsilon_{t_k}|\cF_t\right]}{\PP(\widetilde\varepsilon_{t_k} = 1)} = \frac{\EE\left[L_k \varepsilon_{t_k} |\cF_t\right]}{\PP(\widetilde\varepsilon_{t_k} = 1)} = 
             \frac{\EE\left[\EE\left[L_k \varepsilon_{t_k}|\cF_{t_{k+1}}\right]|\cF_t\right]}{\PP(\widetilde\varepsilon_{t_k} = 1)}\notag \\
             &= \frac{\EE\left[\varepsilon_{t_k} \EE\left[L_k \right]|\cF_t\right]}{\PP(\widetilde\varepsilon_{t_k} = 1)} 
             = p_k \frac{\EE\left[\varepsilon_{t_k} |\cF_t\right]}{\PP(\widetilde\varepsilon_{t_k} = 1)} = p_k \frac{\PP(\varepsilon_{t_k} = 1)}{\PP(\widetilde\varepsilon_{t_k} = 1)}p_t^{1,k} = p_t^{1,k}\ ,
        \end{align}
        we write the differential as
        $d\widetilde p_t^{1,k} = 
        dp_t^{1,k} $
        and we immediately have that
        \begin{equation}\label{eq.tilde.alpha}
             \widetilde\alpha^1_t(k) = 
             \alpha_t^1(k) \ ,
        \end{equation}
        We conclude that the information of $\{\widetilde\varepsilon_{t_k} = 1\}$ is the same as in $\{\varepsilon_{t_k} = 1\}$ because 
        $\{\widetilde\varepsilon_{t_k} = 1\}\subset\{\varepsilon_{t_k} = 1\}$.
        We can then argue in the same way to get~$\widetilde\alpha^0_t(k)$ and we get
        \begin{align*}
             \widetilde p_t^{0,k} &= \frac{\PP(\widetilde\varepsilon_{t_k} = 0|\cF_t)}{\PP(\widetilde\varepsilon_{t_k} = 0)} = \frac{1-\PP(\widetilde\varepsilon_{t_k} = 1|\cF_t)}{\PP(\widetilde\varepsilon_{t_k} = 1)} = \frac{1-p_k\PP(\varepsilon_{t_k} = 1|\cF_t)}{\PP(\widetilde\varepsilon_{t_k} = 1)}\\ 
             &= 
             \frac{1-p_k+p_k\PP(\varepsilon_{t_k} = 0|\cF_t)}{\PP(\widetilde\varepsilon_{t_k} = 1)} =
             \frac{1-p_k}{\PP(\widetilde\varepsilon_{t_k} = 1)}
             +p_k\frac{\PP(\varepsilon_{t_k} = 0)}{\PP(\widetilde\varepsilon_{t_k} = 1)}p_t^{0,k}\ ,
        \end{align*}
        and we write the differential as
        $$d\widetilde p_t^{0,k} = p_k\frac{\PP(\varepsilon_{t_k} = 0)}{\PP(\widetilde\varepsilon_{t_k} = 1)}dp_t^{0,k} $$
        and we conclude
        \begin{align*}
            \frac{d\widetilde p_t^{0,k}}{\widetilde p_t^{0,k}} &= p_k\frac{\PP(\varepsilon_{t_k} = 0)}{\PP(\widetilde\varepsilon_{t_k} = 1)}\frac{dp_t^{0,k}}{\widetilde p_t^{0,k}} = p_k\frac{\PP(\varepsilon_{t_k} = 0)}{\PP(\widetilde\varepsilon_{t_k} = 1)}\frac{ p_t^{0,k}}{ \widetilde p_t^{0,k}}\frac{dp_t^{0,k}}{ p_t^{0,k}} = p_k\frac{\PP(\varepsilon_{t_k} = 0|\cF_t)}{\PP(\widetilde\varepsilon_{t_k} = 1)\widetilde p_t^{0,k}}\frac{dp_t^{0,k}}{ p_t^{0,k}} \\
            &= p_k\frac{\PP(\varepsilon_{t_k} = 0|\cF_t)}{\PP(\widetilde\varepsilon_{t_k} = 1)\widetilde p_t^{0,k}}\alpha^0_t(k) dW_t =: \widetilde\alpha^0_t(k) dW_t \ .
        \end{align*}
        Let~$\widetilde\bG$ be the filtration enlarged with the Markov chain $\widetilde\varepsilon$.
        We aim to compare the additional logarithmic price of the information between~$\widetilde\bG$ and~$\bF$. 
        According to Remark~\ref{rk.gains},~\cite{Fontana18PRP} assures the representation of $\widetilde\varepsilon$, in this case as
        $$\widetilde\varepsilon_{t_k} = L_k\EE[\varepsilon_{t_k}]
        + \int_0^{t_{k+1}} L_k \EE[D_t \varepsilon_{t_k}\vert\cF_t] dW_t\ .$$
        Finally, when we use the~\cite[Theorem 4.1]{AmendingerImkellerSchweizer1998}, we apply the non $\cF_T$-measurable statement and we conclude that
        \begin{align*}
            \bV_T^{\widetilde{\bG}}-\bV_T^\bF =& -\sum_{k=0}^{n-1}\sum_{b\in\{0,1\}}
            \EE\left[\PP(\widetilde\varepsilon_{t_k}=b|\cF_{t_{k+1}})\ln \PP(\widetilde\varepsilon_{t_k}=b|\cF_{t_{k+1}}) 
            \right] \\
            &+ (M_1-M_0)p_k\int_0^T\EE[D_t \varepsilon_{t}] dt\ ,
        \end{align*}
        without imposing $\varepsilon_{t_k}\in\cF_{t_{k+1}}$ for any $k<n$.
        We rewrite it as
        \begin{align*}
        \bV_T^{\widetilde{\bG}}-\bV_T^\bF =& -\sum_{k=0}^{n-1}\EE\left[
        \bOne_{\{\varepsilon_{t_k}=1\}} p_k  
        \ln\left(\bOne_{\{\varepsilon_{t_k}=1\}} p_k\right) +
        (1-\bOne_{\{\varepsilon_{t_k}=1\}} p_k)\ln\left(1-\bOne_{\{\varepsilon_{t_k}=1\}} p_k\right)
        \right]\\
        &+ (M_1-M_0)p_k\int_0^T\EE[D_t \varepsilon_{t}] dt\\
        =&  -\sum_{k=0}^{n-1} \PP(\widetilde\varepsilon_{t_k}=1) \left( p_k\ln p_k + (1-p_k)\ln(1-p_k)\right)+ (M_1-M_0)p_k\int_0^T\EE[D_t \varepsilon_{t}] dt\ .
        \end{align*}
        Finally, we compute the price between~$\bG$ and~$\widetilde\bG$ as follows, 
        \begin{align*}
        \bV_T^\bG-\bV_T^{\widetilde{\bG}} =&  -\sum_{k=0}^{n-1}\left(\PP(\varepsilon_{t_k}=1)\ln \PP(\varepsilon_{t_k}=1) + \PP(\varepsilon_{t_k}=0)\ln \PP(\varepsilon_{t_k}=0)\right)\\
        &+\sum_{k=0}^{n-1} \PP(\widetilde\varepsilon_{t_k}=1) \left( p_k\ln p_k + (1-p_k)\ln(1-p_k)\right) + (M_1-M_0)p_k\int_0^T\EE[D_t \varepsilon_{t}] dt\ ,
        \end{align*}
        where it can be verified that
        $$0<h(x,y):=-(x\ln x +(1-x)\ln(1-x))+xy(y\ln y+(1-y)\ln(1-y))$$ 
        for any $0<x,y<1$,
        while the sign of the last term depends of the sign of $M_1-M_0$ and if $\varepsilon_{t_k}$ is codified as a positive or negative trend of $W$,
        as we can see in the previous examples.
\end{Example}
\section{Incomplete market}\label{subsec:incomplete}
In this section we assume that the Markov chain $\varepsilon$ that modulates the market is measurable in a sigma-algebra bigger than $\cF_T$.
In particular, the market coefficients are also modulated by a Brownian motion $B=\prT{B}$ as follows,
\begin{align}
    dD_t &= {D_t}\,r_{\varepsilon_t}(B_{\cdot})\, dt\ , \quad D_0 = 1 \\
    dS_t &= {S_t}\left(\eta_{\varepsilon_t}(B_{\cdot})\, dt + \xi_{\varepsilon_t}(B_{\cdot})\, dW_t\right)\ ,\quad S_0 = s_0 >0\ .
\end{align}
We assume that $W$ and $B$ are $\PP$-independent and
we introduce the natural information flow $\bE = \{\cE_t, t\geq0\}$ induced by $W$ and $B$ as 
$$\cE_t = \sigma\left( (W_s,B_s), 0\leq s \leq t \right)\ .$$
Note that $(W,B)$ has the predictable representation property in the sense that any $(\PP,\bE)$-local martingale $m=\prT{m}$ can be written as
\begin{equation}\label{PRP.bE}
    m_t = m_0 + \int_0^t \varphi_sdW_s + \int_0^t \phi_s dB_s\ , 
\end{equation}
where $\varphi$, $\phi$ are square-integrable $\bE$-predictable processes.
Indeed, for any $F\in\cE_T$-measurable random variable satisfying $F\in L^2(\PP)$, the following representation holds,
\begin{equation}\label{eq.clark.ocone.G}
    F = \EE[F] + \int_0^T\EE[D^{(1)}_sF\vert\cE_s] dW_s + \int_0^T\EE[D^{(2)}_s F\vert\cE_s] dB_s \ ,
\end{equation}
where $D^{(1)}$ and $D^{(2)}$ represent the Malliavin derivative with respect $W$ and $B$ respectively,
see \cite[Theorem 13.28]{DiNunnoOksendalProske2009} for a reference.
In this case, the discrete time Markov chain satisfies that~$\varepsilon_{t_k}$ is $\cE_{t_{k+1}}$-measurable, being $\bG$ the same enlargement as in \eqref{def.G}.
As we maintain the binary state space, 
the Jacod's hypothesis is still satisfied within the filtration $\bE$.
The motivation is given by assuming the existence of an additional
source of uncertainty in the information flow.
This slight modification induces incompleteness in the market and the computations of the Section~\ref{subsec:complete} do not hold.
\begin{Proposition}\label{prop.BM.G.incomplete}
If we consider the enlargement $\bE\subset\bG$, then 
\begin{align*}
    \widehat W_t &:= W_t - \int_0^t\alpha^\varepsilon_s ds \ ,\quad &\alpha^\varepsilon_t& :=  \sum_{k=0}^{n-1}\bOne_{\{t_k\leq t < t_{k+1}\}}\alpha^\varepsilon_t(k)\ ,\quad &\alpha^e_t(k):=& \frac{\EE[D^{(1)}_t\bOne_{\{\varepsilon_{t_k}=e\}}\vert\cE_t]}{\PP(\varepsilon_{t_k}=e|\cE_t)}\\
    \widehat B_t &:= B_t - \int_0^t\gamma^\varepsilon_s ds \ ,\quad &\gamma^\varepsilon_t& :=  \sum_{k=0}^{n-1}\bOne_{\{t_k\leq t < t_{k+1}\}}\gamma^\varepsilon_t(k)\ ,\quad &\gamma^e_t(k):=& \frac{\EE[D^{(2)}_t\bOne_{\{\varepsilon_{t_k}=e\}}\vert\cE_t]}{\PP(\varepsilon_{t_k}=e|\cE_t)}
\end{align*}
where $\widehat W,\,\widehat B$ are $\bG$-Brownian motions in $[0,T]$. 
\end{Proposition}
\begin{proof}
The result directly follows from Proposition~\ref{prop.jacod.semimartingale} and the 
representation given by \eqref{eq.clark.ocone.G}, 
in particular we get
\begin{align*}
    p_t^{e,k} =& \frac{\PP(\varepsilon_{t_k}=e\vert \cE_t)}{\PP(\varepsilon_{t_k}=e)} 
    =  \frac{\EE[\bOne_{\{\varepsilon_{t_k}=e\}}\vert \cE_t]}{\PP(\varepsilon_{t_k}=e)} 
    \\
    =& \frac{\PP(\varepsilon_{t_k}=e) +  \EE[\int_0^T\EE[D^{(1)}_s\bOne_{\{\varepsilon_{t_k}=e\}}\vert\cE_s] dW_s + \int_0^T\EE[D^{(2)}_s\bOne_{\{\varepsilon_{t_k}=e\}}\vert\cE_s] dB_s \vert \cE_t]}{\PP(\varepsilon_{t_k}=e)}\\
    =& \frac{\PP(\varepsilon_{t_k}=e) +  \int_0^t\EE[D^{(1)}_s\bOne_{\{\varepsilon_{t_k}=e\}}\vert\cE_s] dW_s + \int_0^t\EE[D^{(2)}_s\bOne_{\{\varepsilon_{t_k}=e\}}\vert\cE_s] dB_s }{\PP(\varepsilon_{t_k}=e)}\ .
\end{align*}
Then we write the differential form as
$$\frac{dp_t^{e,k}}{p_t^{e,k}} = \frac{\EE[D^{(1)}_t\bOne_{\{\varepsilon_{t_k}=e\}}\vert\cE_t]}{\PP(\varepsilon_{t_k}=e|\cE_t)} dW_t +\frac{\EE[D^{(2)}_t\bOne_{\{\varepsilon_{t_k}=e\}}\vert\cE_t] }{\PP(\varepsilon_{t_k}=e|\cE_t)} dB_t\ , $$
and by the independence of $W$ and $B$ it follows that
$\langle W,B\rangle_s^{\bE} = 0$ and we get the result.
\end{proof}
\begin{Lemma}\label{lema.PRP.G}
Let $L=\prT{L}$ be any $(\PP,\bG)$-local martingale with $L_0 = 0$. 
Then, there exist square-integrable and $\bG$-predictable processes $\lambda$ and $\nu$ such that
$$ L_t = \int_0^t \lambda_s d\widehat W_s + \int_0^t \nu_s d\widehat B_s \ ,\quad 0 \leq t\leq T\ .  $$
\end{Lemma}
\begin{proof}
The result directly follows from \eqref{PRP.bE}, Proposition \ref{prop.BM.G.incomplete} and \cite[Corollary 2.10]{Fontana18PRP}.
\end{proof}
Motivated by the Lemma \ref{lema.PRP.G}, we introduce the exponential $(\PP,\bG)$-local martingale
$$ Z^{\lambda,\nu}_t := \exp\left( -\int_0^t \lambda_s d\widehat W_s - \int_0^t \nu_s d\widehat B_s -\frac{1}{2}\int_0^t  (\lambda_s^2 +  \nu_s^2) ds \right) \ ,\quad 0 \leq t\leq T\ . $$
\begin{Lemma}\label{lemma.ELMM}
The family of processes $\{Z^\nu\}$ with 
$\nu\in L^2(dt\times\PP,\bG)$ satisfies that
$D^{-1}Z^\nu S$ and $D^{-1}Z^\nu X^{\pi}$ are non-negative $\bG$-local martingales, where we define
$$ Z^{\nu}_t := \exp\left( -\int_0^t \theta^\varepsilon_s d\widehat W_s - \int_0^t \nu_s dB_s -\frac{1}{2}\int_0^t  (\theta^\varepsilon_s)^2 +  \nu_s^2\, ds \right) \ ,\quad 0 \leq t\leq T\ . $$
\end{Lemma}
\begin{proof}
By using It\^o calculus, we get 
\begin{align*}
    Z^{\lambda,\nu}_tD_t^{-1}S_t = S_0\exp\Big\{&\int_0^t \left( \eta_s(B_\cdot) - r_s(B_\cdot) - \frac{1}{2}\xi_s^2(B_\cdot) + \xi_s(B_\cdot)\alpha_s^\varepsilon -\frac{1}{2}(\lambda_s^2+\nu_s^2)\right) ds \\
    &+\int_0^t \left(\xi_s(B_\cdot) - \lambda_s \right) d\widehat W_s
    -\int_0^t \nu_s  dB_s
    \Big\}\ ,
\end{align*}
then we differentiate in It\^o sense,
\begin{align*}
    \frac{d(Z^{\lambda,\nu}_tD_t^{-1}S_t)}{Z^{\lambda,\nu}_tD_t^{-1}S_t} 
    =&  \xi_t(B_\cdot)\left( \frac{\eta_t(B_\cdot) - r_t(B_\cdot)}{\xi_t(B_\cdot)}  +\alpha_t^\varepsilon  - \lambda_t\right) dt + \left(\xi_t(B_\cdot) - \lambda_t \right) d\widehat W_t
     - \nu_t  dB_t\ .
\end{align*}
By imposing in the previous expression that the drift part is null,
we obtain that $\lambda_t=\theta^\varepsilon_t$. 
\end{proof}
The Inada conditions \eqref{inada.conditions} assure that the function~$I$ is such that~$I(0+) = \infty$ and~$I(\infty) = 0$.
We define
$$ \widehat U (y) := \max_{x > 0} \{U(x) - xy\} = U(I(y)) - yI(y)\ ,\quad 0 <  y <  +\infty\ ,$$
and formulate some assumptions appearing in \cite{karatzas1991} (see also \cite{PhamQuenez2001}).
The first one excludes the logarithmic utility, 
but in this case, the optimization problem can be solved explicitly.
\begin{Assumption}\label{Assump.U}
  \begin{itemize}
      \item $U(0) := \lim_{x\to 0+} U(x) > -\infty .$
      \item The function $h(x):= x U'(x)$ is non decreasing on $\bR^+$.
      \item There exist $\alpha\in (0,1)$ and $\beta\in(1,+\infty)$, such that $ \alpha U'(x) \geq U'(\beta x)$, for every~\hbox{$x\in\bR^+$}.
      \item There exists 
      $\nu\in L^2(dt\times\PP,\bG)$
      such that \hbox{$\EE[\widetilde U (y D_T^{-1} Z_T^\nu)]<+\infty$} for any $y\in\bR^+$.
  \end{itemize}
\end{Assumption}
So, the dual problem of \eqref{utility} is stated as follows
\begin{equation}\label{problem.dual}
    \widetilde{V}(y) := \inf_{\nu} 
    \EE[\widetilde{U}(y D_T^{-1}Z_T^\nu)]\ .
\end{equation}
Finally, we introduce the following family of random variables
\begin{equation}\label{def.xi}
    \xi^{\nu}(x_0) := I\left( \cY^\nu(x_0) D^{-1}_T Z^\nu_T \right)\ ,\quad \nu\in L^2(dt\times\PP,\bG)\ ,
\end{equation}
where $\cY^\nu(x_0)$ satisfies
\begin{equation}\label{def.Y.cal}
    \EE\left[ D_T^{-1} Z^\nu_T I\left( \cY^\nu(x_0) D_T^{-1} Z^\nu_T\right)\right] = x_0\ .
\end{equation}
The following result, proved in \cite{karatzas1991}, 
gives the existence of a deflator under which the contingent claim \eqref{def.xi} is hedgeable and the expectation is preserved.
\begin{Proposition}\label{prop.karatzas}
Under Assumption \ref{Assump.U}, there exists  $\lambda$ such that
for any~$\nu$
\begin{equation}\label{dual.opt}
    \EE\left[D_T^{-1}Z_T^\nu \xi^{\lambda}({x_0})\right] \leq x_0 = \EE\left[D_T^{-1}Z_T^{\lambda}\xi^{\lambda}(x_0)\right]\ . 
\end{equation}
Moreover, the random variable $\xi^{\lambda}(x_0)$ is hedgeable and 
$\lambda$ is the solution of the problem~\eqref{problem.dual} with parameter $\cY^\lambda(x_0)$.
\end{Proposition}
\begin{proof}
It is a consequence of Theorem 8.5, Theorem 9.4 and Theorem 12.3 of \cite{karatzas1991}.
\end{proof}
\begin{Theorem}\label{theo.opt.incomplete}
The solution of the optimal portfolio problem is 
\begin{align}\label{opt.utility}
  \esssup_{\pi\in\cA(\bG)} 
  \EE[U( X^{\pi}_T)] =  \EE[U(\xi^\lambda(x_0))] \ ,
\end{align}
where $X^\pi$ satisfies~\eqref{X.SDE}.
The optimal portfolio is attained by a strategy {$\pi^{\bG}\in\cA(\bG)$} which hedges the random variable~\hbox{$\xi^\lambda(x_0) = I(\cY^\lambda(x_0) D_T^{-1} Z^\lambda_T)$}
where $\lambda$ is defined in Proposition \ref{prop.karatzas}.
\end{Theorem}
\begin{proof}
Using Proposition \ref{prop.karatzas} we know that there exists $\pi^\bG\in\cA(\bG)$ such that 
$X^{\pi^\bG}_T = \xi^\lambda_T(x_0)$ 
$\PP$-almost surely.
The rest of the proof is analogous to the complete case, see Proposition \ref{theo.opt.complete}.
\end{proof}
\begin{Example}
We define 
$$ \varepsilon_{t_k} = \bOne_{\{W_{t_{k+1}} > W_{t_k}\}}\bOne_{\{B_{t_{k+1}} > B_{t_k}\}}\ ,\quad k\in\{0,1,...,n-1\}  \ .$$
We aim to compute the information drifts of $W$ and $B$ using Proposition \ref{prop.BM.G.incomplete}.
We proceed as follows,
\begin{align*}
    \alpha^1_t(k)=& \frac{\EE[D^{(1)}_t\bOne_{\{\varepsilon_{t_k}=1\}}\vert\cE_t]}{\PP(\varepsilon_{t_k}=1|\cE_t)} =  \frac{\EE[\bOne_{\{B_{t_{k+1}} > B_{t_k}\}}D^{(1)}_t\bOne_{\{W_{t_{k+1}} > W_{t_k}\}}
    \vert\cE_t]}{\PP(\varepsilon_{t_k}=1|\cE_t)}\\ 
    =& 
    \frac{\EE[\bOne_{\{B_{t_{k+1}} > B_{t_k}\}}\delta_{W_{t_k}}(W_{t_{k+1}})
    \vert\cE_t]}{\PP(\varepsilon_{t_k}=1|\cE_t)}
    =  \frac{\EE[\EE[\bOne_{\{B_{t_{k+1}} > B_{t_k}\}}\delta_{W_{t_k}}(W_{t_{k+1}})
    \vert \cE_t\vee\cF_{t_{k+1}}]
    \vert\cE_t]}{\PP(\varepsilon_{t_k}=1|\cE_t)} \\
    =& \frac{\EE[\delta_{W_{t_k}}(W_{t_{k+1}})\EE[\bOne_{\{B_{t_{k+1}} > B_{t_k}\}}
    \vert \cE_t]\vert\cE_t]}{\PP(\varepsilon_{t_k}=1|\cE_t)} = \overline\Phi\left(\frac{B_{t_k}-B_t}{\sqrt{t_{k+1}-t}}\right)\frac{1}{\sqrt{t_{k+1}-t}}\frac{\Phi'\left(\frac{W_{t_k}-W_t}{\sqrt{t_{k+1}-t}}\right)}{\PP(\varepsilon_{t_k}=1|\cE_t)} 
    \end{align*}
    Similarly, we compute
\begin{equation*}
    \gamma^1_t(k)= \frac{\EE[D^{(2)}_t\bOne_{\{\varepsilon_{t_k}=1\}}\vert\cE_t]}{\PP(\varepsilon_{t_k}=1|\cE_t)} = \overline\Phi\left(\frac{W_{t_k}-W_t}{\sqrt{t_{k+1}-t}}\right)\frac{1}{\sqrt{t_{k+1}-t}}\frac{\Phi'\left(\frac{B_{t_k}-B_t}{\sqrt{t_{k+1}-t}}\right)}{\PP(\varepsilon_{t_k}=1|\cE_t)} \ .
\end{equation*}
As we have pointed out, the logarithmic utility does not satisfy the Assumption \ref{Assump.U}.
However, we can maximize handling with the good properties of the logarithmic, as in Example \ref{Example.BM.increment}.
We conclude that the optimal strategy is
\begin{equation}
        \pi^\bG_t =  \frac{\eta_{\varepsilon_t}}{\xi^2_{\varepsilon_t}} +\frac{\alpha^{\varepsilon}_t}{\xi_{\varepsilon_t}} =
        \sum_{k=0}^{n-1}\bOne_{\{k\leq t < k+1\}}\left( \frac{\eta_{\varepsilon_{t_k}}}{\xi^2_{\varepsilon_{t_k}}} +\frac{\alpha^{\varepsilon}_{t_k}(k)}{\xi_{\varepsilon_{t_k}}} \right)\ ,
\end{equation}
    which is depending only on $\alpha^{\varepsilon}_t(k)$ and not on $\gamma^{\varepsilon}_t(k)$.
\end{Example}
\section{Conclusion}
In this paper we have shown how to incorporate the knowledge about an anticipative 
market modulation in the information flow of an agent who invests in a two-asset market.
In this set-up, we have solved the expected utility optimization problem for both
the complete and the incomplete market,
using the techniques appearing in \cite{karatzas1991}. 
We have also given the gain of the anticipative information 
for the logarithmic utility.
Among the examples we present, we construct one
where the information contains an additional noise made
by a sequence of Bernoulli independent random variables, i.e, may be wrong. 
We show that under this framework the market remains complete and that the accurate information assures higher expected profits under a condition on the market coefficient.